\documentclass{amsart}

\usepackage{amssymb}
\usepackage{amsmath}
\usepackage{amsthm}
\usepackage{mathrsfs}
\usepackage{dsfont}
\usepackage{url}

\newtheorem*{theorem}{Theorem}
\newtheorem{lemma}{Lemma}
\newtheorem*{corollary}{Corollary}
\newtheorem{jumping-champion-conjecture}{Conjecture}
\newtheorem*{conjecture-a}{Conjecture A}
\newtheorem*{conjecture-b}{Conjecture B}

\newcommand{\abs}[1]{\lvert#1\rvert}

\begin{document}

\title[The jumping champion conjecture]{The jumping champion conjecture}


\author[Daniel A. Goldston]{Daniel A. Goldston$^{*}$}
\address{Department of Mathematics, San Jos\'{e} State University, 315 MacQuarrie Hall, One Washington Square, San Jos\'{e}, California 95192-0103, United States of America}
\email{daniel.goldston@sjsu.edu}
\thanks{$^{*}$During the preparation of this work, the first author received support from the National Science Foundation Grant DMS-0804181.}

\author[Andrew H. Ledoan]{Andrew H. Ledoan}
\address{Department of Mathematics, Boston College, 228 Carney Hall, 140 Commonwealth Avenue, Chestnut Hill, Massachusetts 02467-3806, United States of America}
\curraddr{Department of Mathematics, The University of Tennessee at Chattanooga, 417F EMCS Building (Department 6956), 615 McCallie Avenue, Chattanooga, Tennessee 37403-2598, United States of America}
\email{andrew-ledoan@utc.edu}

\subjclass[2010]{Primary 11N05; Secondary 11P32, 11N36.}



\keywords{Differences between consecutive primes; Hardy-Littlewood prime $k$-tuple conjecture; jumping champion; maximal prime gaps; primorial numbers; sieve methods; singular series.}

\begin{abstract}
An integer $d$ is called a jumping champion for a given $x$ if $d$ is the most common gap between consecutive primes up to $x$. Occasionally several gaps are equally common. Hence, there can be more than one jumping champion for the same $x$. For the $n$th prime $p_{n}$, the $n$th primorial $p_{n}^{\sharp}$ is defined as the product of the first $n$ primes. In 1999, Odlyzko, Rubinstein and Wolf provided convincing heuristics and empirical evidence for the truth of the hypothesis that the jumping champions greater than 1 are 4 and the primorials $p_{1}^{\sharp}, p_{2}^{\sharp}, p_{3}^{\sharp}, p_{4}^{\sharp}, p_{5}^{\sharp}, \ldots$, that is, $2, 6, 30, 210, 2310, \ldots.$ In this paper, we prove that an appropriate form of the Hardy-Littlewood prime $k$-tuple conjecture for prime pairs and prime triples implies that all sufficiently large jumping champions are primorials and that all sufficiently large primorials are jumping champions over a long range of $x$.
\end{abstract}

\maketitle


\thispagestyle{empty}

\section{Introduction and statement of results}

In 1993, John Horton Conway invented the term jumping champion to refer to the most common gap between consecutive primes not exceeding $x$. For the $n$th prime $p_{n}$, the jumping champions are the values of the integer $d$ for which the counting function
\[
N(x, d)
 = \sum_{\substack{p_{n} \leq x \\ p_{n} - p_{n - 1} = d}} 1
\]
attains its maximum
\[
N^{*}(x)
 = \max_{d} N(x, d).
\]
Thus, the set of jumping champions for primes not exceeding $x$ is defined by
\[
\mathscr{D}^{*}(x)
 = \{ d^{*} \colon N(x, d^{*}) = N^{*}(x)\}.
\]
Furthermore, the $n$th primorial $p_{n}^{\sharp}$ is defined as the product of the first $n$ primes.

In 1999, Odlyzko, Rubinstein and Wolf \cite{OdlyzkoRubinsteinWolf1999} enunciated the following hypothesis, now usually known as the jumping champion conjecture.

\begin{jumping-champion-conjecture} \label{first-jumping-champion-conjecture}
The jumping champions greater than 1 are 4 and the primorials $p_{1}^{\sharp}, p_{2}^{\sharp}, p_{3}^{\sharp}, p_{4}^{\sharp}, p_{5}^{\sharp}, \ldots$, that is, $2, 6, 30, 210, 2310, \ldots.$
\end{jumping-champion-conjecture}

Odlyzko, Rubinstein and Wolf also introduced a weaker hypothesis that follows from Conjecture \ref{first-jumping-champion-conjecture}.

\begin{jumping-champion-conjecture} \label{second-jumping-champion-conjecture}
The jumping champions tend to infinity. Furthermore, any fixed prime $p$ divides all sufficiently large jumping champions.
\end{jumping-champion-conjecture}

The first assertion of Conjecture \ref{second-jumping-champion-conjecture} was proved by Erd\H{o}s and Straus \cite{ErdosStraus1980} in 1980, under the assumption of the truth of the Hardy-Littlewood prime pair conjecture. Very recently, we have extended their method to give a complete proof of Conjecture \ref{second-jumping-champion-conjecture}, again under the same assumption. (See Goldston and Ledoan \cite{GoldstonLedoan2011}.)

In the present paper, we consider Conjecture \ref{first-jumping-champion-conjecture} with the aid of the Hardy-Littlewood prime $k$-tuple conjecture. It is clear that in order to determine completely the jumping champions one requires information about prime triples, in addition to a method for eliminating larger than average gaps between primes from being jumping champions. The solution to the latter problem, as simple as it turned out to be, is fundamental to the success of our work.

We defer stating precisely the Hardy-Littlewood prime $k$-tuple conjecture for prime pairs (that is, Conjecture A) and prime triples (that is, Conjecture B) until the next section. Our main result can be summarized as follows, its details as described in the original work by Odlyzko, Rubinstein and Wolf.

\begin{theorem}
Assume Conjectures A and B. Let $a$, $a'$, $b$ and $b'$ be any fixed numbers  satisfying $3 / 4 \leq a' < b < 1 < a < b' \leq 5 / 4$ and let $x$ be sufficiently large. Then the interval
\[
\left[\frac{a \log x}{(\log\log x)^2}, \frac{b \log x}{\log\log x}\right]
\]
can contain at most one primorial. If this interval contains a primorial, then this primorial will be the jumping champion for $x$. If this interval does not contain a primorial, then the two intervals
\[
\left[\frac{a' \log x}{(\log\log x)^{2}}, \frac{a\log x}{(\log\log x)^{2}}\right)
\quad \mbox{and} \quad
\left(\frac{b \log x}{\log\log x}, \frac{b'\log x}{\log\log x}\right]
\]
will each contain a primorial. Furthermore, one or the other and sometimes both of these primorials will be the jumping champion for $x$.
\end{theorem}

We can reformulate the theorem in terms of a range of $x$ over which a primorial is the jumping champion.

\begin{corollary}
Assume Conjectures A and B. Then for any $0 < \delta < 1 / 2$ and $k$ sufficiently large, the $k$th primorial $p_{k}^{\sharp}$ is the jumping champion for all $x$ in the interval
\[
[\exp((1 + \delta) p_{k}^{\sharp} \log p_{k}^{\sharp}), \exp((1 - \delta) p_{k}^{\sharp}(\log p_{k}^{\sharp})^{2})].
\]
\end{corollary}

For obviously
\[
(1 + \delta) p_{k}^{\sharp} \log p_{k}^{\sharp}
 \leq \log x
 \leq (1 - \delta) p_{k}^{\sharp} (\log p_{k}^{\sharp})^{2},
\]
from which follows that
\[
\log p_{k}^{\sharp}
 \sim \log \log x, \quad \mbox{as $x \to \infty$}.
\]
Thus, for $x$ sufficiently large,
\[
\left(1 + \frac{\delta}{2}\right) \frac{\log x}{(\log\log x)^{2}}
 \leq p_{k}^{\sharp}
 \leq \left(1 - \frac{\delta}{2}\right) \frac{\log x}{\log\log x}.
\]
Hence, the corollary.

Several authors have conducted extensive computations to determine the precise point of transition between jumping champions. For example, it was computed by Harley \cite{Harley1994} in 1994 that the transition from 6 to 30 should occur at $1.7 \cdot 10^{35}$, while the heuristics of Odlyzko, Rubinstein and Wolf suggest that 6 is the jumping champion up to about $1.7427 \cdot 10^{35}$.

Our results are not very helpful in this regard. Several intervals from the corollary with $\delta = 0$ are indicated in Table \ref{table_one} below. Even though these intervals may have the corresponding primorial as their jumping champion, they produce an enormous range for the transition zone.

\begin{table}[ht] \label{table_one} \vspace*{0ex}
\caption{Intervals where $p_{k}^{\sharp}$ is a likely jumping champion.}
\begin{tabular}{|rrr|}
\hline
$k$	&	$p_{k}^{\sharp}$	&	$[\exp(p_{k}^{\sharp} \log p_{k}^{\sharp}) , \exp(p_{k}^{\sharp} (\log p_{k}^{\sharp})^{2})]$	\\
\hline
$3$	&	6				&	$[4.67 \cdot 10^{4}, 2.32 \cdot 10^{8}]$				\\
$4$	&	30				&	$[2.06 \cdot 10^{44}, 5.24 \cdot 10^{150}]$			\\
$5$	&	210				&	$[4.64 \cdot 10^{487}, 4.01 \cdot 10^{2607}]$			\\
$6$	&	2310				&	$[8.78 \cdot 10^{7769}, 1.72 \cdot 10^{60178}]$		\\
$7$	&	30030			&	$[9.70 \cdot 10^{134460}, 1.72 \cdot 10^{1386286}]$	\\
\hline
\end{tabular}
\end{table}

There are several directions in which our method can be extended. One could, for example, employ more refined versions of Conjectures A and B and strive to be more precise about the transition zones. Likewise, one could presumably formulate explicit versions of these two conjectures substantiated by numerical experimentation and use them to derive a complete proof of Conjecture \ref{first-jumping-champion-conjecture}.

\section{The Hardy-Littlewood prime $k$-tuple conjectures}

In a pioneering paper published in 1923, Hardy and Littlewood \cite{HardyLittlewood1923} created and developed a new analytic method in additive number theory to attack Goldbach's conjecture. They were led to infer an asymptotic formula for counting prime tuples, which may be stated as follows: Let $\mathcal{D} = \{d_{1}, \ldots, d_{k}\}$ be a set of $k$ distinct integers and let $\pi(x; \mathcal{D})$ denote the number of positive integers $n \leq x$ such that $n + d_{1}, \ldots, n + d_{k}$ are all prime. Define
\[
\mbox{\textup{li}}_{k}(x)
 = \int_{2}^{x} \frac{\,dt}{(\log t)^{k}}
\]
and
\[
\mathfrak{S}(\mathcal{D})
 = \prod_{p}\left(1 - \frac{1}{p}\right)^{-k} \left(1 - \frac{\nu_{\mathcal{D}}(p)}{p}\right),
\]
where $p$ runs through all the primes, and $\nu_{\mathcal{D}}(p)$ represents the number of distinct residue classes modulo $p$ occupied by the elements of $\mathcal{D}$. If $\mathfrak{S}(\mathcal{D}) \neq 0$, then
\[
\pi(x; \mathcal{D})
 \sim \mathfrak{S}(\mathcal{D}) \ \! \mbox{\textup{li}}_{k}(x), \quad \mbox{ as $x \to \infty$}.
\]
If now $\nu_{\mathcal{D}}(p) = p$ for some $p$, then $\mathfrak{S}(\mathcal{D}) = 0$, in which case $\pi(x; \mathcal{D})$ will equal to either 0 or 1.

The asymptotic formula above has been verified only for the prime number theorem, that is, for the case of $k = 1$. It has been conjectured that, in its strongest form, the formula holds true for any fixed $k$ with an error term that is $O_{k}(x^{1 / 2 + \varepsilon})$ at most and uniformly for $\mathcal{D} \subset [1, x]$.

However, we do not need such strong error terms and uniformity. As it happens, an error term of order only slightly smaller than the first term in the asymptotic expansion of $\mbox{\textup{li}}_{k}(x)$ is quite sufficient for our present purpose. Furthermore, in counting consecutive prime gaps for jumping champions, we consider only those gaps that do not exceed $x$, rather than employing the usual counting method for $\pi(x,\mathcal{D})$ which takes account of tuples that begin with positive integers $n \leq x$, but can exceed $x$.

The asymptotic formulas involved, naturally, are indifferent to which counting method we adopt. Therefore, we may alter slightly the form of the Hardy-Littlewood prime pair and prime triple conjectures as follows.

\begin{conjecture-a}
Let $d$ be any positive integer and let $p$ be a prime. Let, further, $\mathfrak{S}(d) = \mathfrak{S}(\{0, d\})$ be the singular series given explicitly by
\[
 \mathfrak{S}(d) = \left\{ \begin{array}{ll}
      {\displaystyle 2 C_{2}\prod_{\substack{p \mid d \\ p > 2}} \left(\frac{p - 1}{p - 2}\right),} & \mbox{if $d$ is  even;} \\
      0,   & \mbox{if $d$ is  odd;} \\
\end{array}
\right.
\]
where
\[
C_{2}
 = \prod_{p > 2}\left(1 - \frac{1}{(p - 1)^{2}}\right)
 = 0.66016 \ldots.
\]
Define
\[
\pi_2(x, d)
 = \sum_{\substack{p \leq x \\ p - p' = d}} 1,
\]
where the sum is taken over primes $p \leq x$ such that $p - p' = d$, and $p'$ is a previous prime before $p$ but not necessarily adjacent to $p$. Then we have
\[
\pi_2(x, d)
 = \mathfrak{S}(d) \frac{x}{(\log x)^{2}}\left[1 + o\left(\frac{1}{(\log\log x)^{2}}\right)\right], \quad \mbox{as $x \to \infty$},
 \]
uniformly for positive even integers $d$ satisfying $2 \leq d \leq (\log x)^{2}$.
\end{conjecture-a}

\begin{conjecture-b}
Let
\[
\pi_{3}(x, d', d)
 = \sum_{\substack{p \leq x \\ p - p' = d \\ p - p'' = d'}} 1.
\]
Then we have
\[
\pi_{3}(x, d', d)
 = \mathfrak{S}(\{0, d', d\}) \frac{x}{(\log x)^{3}} (1 + o(1)), \quad \mbox{as $x \to \infty$},
\]
when $\mathfrak{S}(\{0, d', d\}) \neq 0$, uniformly for positive even integers $d$ and $d'$ satisfying $2 \leq d' < d \leq (\log x)^{2}$.
\end{conjecture-b}

It is to be understood, here and in all that follows, that $d$ represents a positive even integer. Finally, we shall have frequent recourse to the well-known sieve bound, for $x$ sufficiently large,
\begin{equation} \label{equation-one}
\pi(x; \mathcal{D})
 \leq (2^{k} k! + \varepsilon) \mathfrak{S}(\mathcal{D}) \frac{x}{(\log x)^{k}},
\end{equation}
when $\mathfrak{S}(\mathcal{D}) \neq 0$. (See Halberstam and Richert's excellent monograph \cite{HalberstamRichert1974}.) This will be used specifically to examine the contribution from quadruples of primes.

\section{Inclusion-exclusion for consecutive primes}

The properties of $N(x, d)$ needed for the proof of the theorem are embodied in the following lemma.

\begin{lemma} \label{lemma-one}
Assume Conjectures A and B. Let $d$ be a positive even integer.

\begin{enumerate}
\item[(i)] If $2 \leq d \leq (\log x)^{2} $, then we have
\begin{equation} \label{equation-two}
N(x,d)
 \leq \mathfrak{S}(d) \frac{x}{(\log x)^{2}} \left[1 + o\left(\frac{1}{(\log\log x)^{2}}\right)\right].
\end{equation}

\item[(ii)] If $2 \leq d \leq o(\log x)$, then we have
\begin{equation} \label{equation-three}
N(x,d)
 = \mathfrak{S}(d) \frac{x}{(\log x)^{2}} \left[1 - \frac{d}{\log x} + o\left(\frac{d}{\log x}\right)+o\left(\frac{1}{(\log\log x)^{2}}\right)\right].
\end{equation}

\item[(iii)] If $D \leq d \leq (\log x)^{2}$ for any $D$ in the interval $\log x / \log\log x \leq D \leq (\log x) / 200$, then we have
\begin{equation} \label{equation-four}
N(x,d)
 \leq \mathfrak{S}(d) \frac{x}{(\log x)^{2}}\left[1 - \frac{D}{\log x} (1 + o(1)) + 200 \left(\frac{D}{\log x}\right)^{2} \right].
\end{equation}
\end{enumerate}
\end{lemma}

\begin{proof}
To prove part (i), observe that
\begin{equation} \label{equation-five-trivial-bound}
N(x,d)
 \leq \pi_{2} (x, d).
\end{equation}
Then \eqref{equation-two} follows at once from Conjecture A.

To prove part (ii), suppose that $p$ and $p'$ are not consecutive primes in $\pi_2(x,d)$. Then there exists a third prime $p''$ such that $p' < p'' < p$. Writing $d' = p - p''$, we then exclude these pairs of primes to obtain the lower bound
\begin{equation} \label{equation-six-lower-bound}
N(x,d)
 \geq \pi_{2}(x, d) - \sum_{1 \leq d' < d} \pi_{3}(x, d', d).
\end{equation}
Of course, if there was a fourth prime between $p$ and $p'$, then we have removed this pair twice, requiring inclusion as the next step, and so on.

We now define the counting function
\[
\pi_{k}(x, d_{1}, \ldots, d_{k - 2}, d)
 = \sum_{\substack{ p \leq x \\ p - p' = d \\ p - p_{j} = d_{j} \\ 1 \leq j \leq k - 2}} 1,
\]
for any $k \geq 3$. As a result, we have
\[
N(x, d)
 \leq \pi_{2} (x, d) - \sum_{1 \leq d' < d} \pi_{3} (x, d', d) + \sum_{1 \leq d_{1} < d_{2} < d} \pi_{4}(x, d_{1}, d_{2},d).
\]

We do not need to carry the inclusion-exclusion any further. However, we do require a trivial modification of this last inequality. We can obtain upper bounds for $N(x, d)$ by excluding only some triples of primes, as long as we compensate for the overcount by including the corresponding quadruples of primes.

Thus, we have
\begin{equation} \label{equation-seven-upper-bound}
N(x, d)
 \leq \pi_{2}(x, d) - \sum_{1 \leq d' < D} \pi_{3}(x, d', d) + \sum_{1 \leq d_{1} < d_{2} < D} \pi_{4}(x, d_{1}, d_{2}, d),
\end{equation}
if $1 \leq D \leq d$. The advantage of this modified inequality is that we can keep $D$ small enough to ensure that the number of quadruples of primes counted does not overwhelm the triples of primes excluded, thereby obtaining a nontrivial upper bound for $N(x, d)$.

We shall now prove \eqref{equation-three} when $d$ is confined to the range $2 \leq d \leq \log x / (\log\log x)^{5}$. By \eqref{equation-five-trivial-bound} and \eqref{equation-six-lower-bound}, we have
\[
N(x, d)
 = \pi_{2}(x, d) + O\left(\sum_{1 \leq d' < d} \pi_3(x, d', d)\right).
\]
Applying \eqref{equation-one} with $k = 3$ and the estimate $\mathfrak{S}(\{0,d',d\}) = O((\log d)^{2})$ (see (3.3) as proved in Section 4 of Goldston and Ledoan \cite{GoldstonLedoan2011}, or see \eqref{equation-sixteen-Singulargoodbound} below for a sharper estimate), we obtain
\begin{align*}
\sum_{1 \leq d' < d} \pi_{3}(x, d', d)
 &= O\left(d(\log d)^{2} \frac{x}{(\log x)^{3}}\right) \\
 &= O\left(\frac{x}{(\log x)^{2} (\log\log x)^{3}}\right).
\end{align*}
Hence, by Conjecture A and the inequalities $\mathfrak{S}(d) \geq 2 C_{2} > 1$ valid for positive even integers $d$, we have
\[
N(x, d)
 = \mathfrak{S}(d) \frac{x}{(\log x)^{2}} \left[1 + o\left(\frac{1}{(\log\log x)^{2}}\right)\right],
\]
if $2\leq d \leq \log x / (\log\log x)^{5}$.

Next, we prove the lemma in the more delicate range $d \geq \log x / (\log\log x)^{5}$. To this end, we have to appeal to the result below about the average of singular series, which has been established for the case of $D = d$ as Theorem 3 in Section 3 of Odlyzko, Rubinstein and Wolf \cite{OdlyzkoRubinsteinWolf1999}. The result is interesting in itself and we therefore include its proof in the last section.

\begin{lemma} \label{lemma-two}
Let $k\geq 3$ and let $d^{\varepsilon} \leq D \leq d$, where $d$ is a positive even integer. Define the singular series average
\[
A_{k}(d, D)
 = \sum_{1 \leq d_{1} < \ldots < d_{k - 2} < D} \mathfrak{S}(\{0, d_{1}, \ldots, d_{k - 2},d\}).
\]
Then we have
\[
A_{k}(d, D)
 = \mathfrak{S}(d) \frac{D^{k - 2}}{(k - 2)!} (1 + o_{k}(1)), \quad \mbox{as $D\to \infty$}.
\]
\end{lemma}

To prove part (iii), we assume that $\log x / (\log\log x)^{5} \leq D \leq d \leq (\log x)^{2}$. Applying Conjectures A and B and Lemma \ref{lemma-two} with $k = 3$, we find that
\begin{align*}
&\pi_{2}(x, d) - \sum_{1 \leq d' < D} \pi_{3}(x, d', d) \\
&\hspace{1.5cm} = \mathfrak{S}(d) \frac{x}{(\log x)^{2}}\left[1 + o\left(\frac{1}{(\log\log x)^{2}}\right)\right] - A_{3}(d, D)\frac{x}{(\log x)^{3}}(1 + o(1)) \nonumber \\ 
&\hspace{1.5cm} = \mathfrak{S}(d) \frac{x}{(\log x)^{2}}\left[1 - \frac{D}{\log x} (1 + o(1)) + o\left(\frac{1}{(\log\log x)^{2}}\right)\right], \quad \mbox{as $x\to \infty$}.
\end{align*}
By using \eqref{equation-one} and Lemma \ref{lemma-two} with $k = 4$, we are led to
\begin{align*}
\sum_{1 \leq d_{1} < d_{2} < D} \pi_{4}(x, d_{1}, d_{2}, d)
 &\leq (384 + \varepsilon) A_{4}(d, D) \frac{x}{(\log x)^{4}} \\
 &< 200 \mathfrak{S}(d)\left(\frac{D}{\log x}\right)^2 \frac{x}{(\log x)^{2}}.
\end{align*}
Here, the terms with $\mathfrak{S}(\{0, d_1, d_2, d\}) = 0$ have $\pi_{4}(x, \mathcal{D}) = 0$ or 1 and contribute $O(D^{2})$, which is absorbed into the $\varepsilon$. If we combine these two results in \eqref{equation-seven-upper-bound}, we obtain \eqref{equation-four} in the ranges stated for $D$ and $d$.

We now finish the proof of part (ii). We set $D = d$ in the last two equations and observe that
\[
\sum_{1 \leq d_{1} < d_{2} < d} \pi_{4}(x, d_{1}, d_{2}, d)
 = o\left(\frac{d}{\log x} \mathfrak{S}(d) \frac{x}{(\log x)^2}\right),
\]
if $2 \leq d = o(\log x)$. Then \eqref{equation-three} follows from substituting these results into \eqref{equation-six-lower-bound} and \eqref{equation-seven-upper-bound}. This finishes the proof of the lemma.
\end{proof}

\section{Proof of the theorem}

From Lemma \ref{lemma-one}, it is not difficult to see heuristically that the jumping champion for $x$ should be a primorial between the two values $\log x / (\log\log x)^2$ and $\log x / \log\log x$. When the primorial is very close to either one of these values, there inevitably will be a second primorial near the other value, and the jumping champion will be one or the other or both of these primorials.

The property of the singular series $\mathfrak{S}(d)$ most crucial for jumping champions is that $\mathfrak{S}(d)$ increases most rapidly on the sequence of primorials. This is clear from the formula
\[
\mathfrak{S}(d)
 = 2 C_{2} \prod_{\substack{p \mid d \\ p > 2}} \left(1+ \frac{1}{p - 2}\right),
\]
for even integers $d$, which leads us to the inequality $\mathfrak{S}(d) < \mathfrak{S}(p_{k}^{\sharp})$, valid for every integer $d$ in the range $2 \leq d < p_{k}^{\sharp}$. (See Lemma 2.2 in Goldston and Ledoan \cite{GoldstonLedoan2011}.)

The primorials have a very simple distribution governed by the prime number theorem which, from the alternative formulation
\[
\vartheta(x)
 \sim x, \quad \mbox{as $x \to \infty$},
\]
where
\[
\vartheta(x)
 = \sum_{p \leq x}\log p
\]
(see Ingham's classical tract \cite{Ingham1932}, Theorem 3, Formula (6), p. 13; and Montgomery and Vaughan's monograph \cite{MontgomeryVaughan2007}, Corollary 2.5, p. 49), gives us
\[
\log p_{k}^{\sharp}
 = \vartheta(p_{k})
 \sim p_{k}, \quad \mbox{as $p_{k} \to \infty$},
\]
from which we deduce that
\[
p_{k}^{\sharp}
 = p_{k} p_{k - 1}^{\sharp}
 \sim p_{k - 1}^{\sharp} \log p_{k}^{\sharp}, \quad \mbox{as $k \to \infty$}.
\]

Before proceeding further, we need to make a simple observation. We have
\[
\log p_{k}^{\sharp} \sim \log\log x, \quad \mbox{as $k \to \infty$},
\]
whence
\begin{equation} \label{equation-eight-pkloglogx}
p_{k} \sim \log\log x \quad \mbox{and} \quad
p_{k}^{\sharp}
 \sim p_{k - 1}^{\sharp} \log\log x, \quad \mbox{as $k \to\infty$},
\end{equation}
if $(\log x)^{1 - o(1)} \leq p_{k}^{\sharp} \leq (\log x)^{1 + o(1)}$.

Our theorem is a simple deduction from Lemmas \ref{lemma-three}, \ref{lemma-four} and \ref{lemma-five} below. Before establishing these results, we retain the notation of the proof of Theorem 2.1 in Goldston and Ledoan \cite{GoldstonLedoan2011} for the floor function with respect to a given increasing sequence $\{a_{j}\}_{j = 1}^{\infty}$ and introduce a similar notation for the ceiling function. We write $\lfloor y \rfloor_{a_{j}} = a_{n}$ if $a_{n} \leq y < a_{n + 1}$, and $\lceil y \rceil_{a_{j}} = a_{n}$ if $a_{n - 1} < y \leq a_{n}$. Hence, by $\lfloor y \rfloor_{p_{j}^{\sharp}}$ we mean the largest primorial not exceeding $y$, and by $\lceil y \rceil_{p_{j}^{\sharp}}$ the smallest primorial greater than or equal to $y$.

\begin{lemma} \label{lemma-three}
Assume Conjectures A and B. For any $0 < \delta \leq 1 / 4$, let
\[
p_{k}^{\sharp}
 = \left\lfloor \frac{(1 - \delta) \log x}{\log\log x} \right\rfloor_{p_{j}^{\sharp}}.
\]
Then for $x$ sufficiently large and $2 \leq d < p_{k}^{\sharp}$, we have
\[
N(x, p_{k}^{\sharp})
 > N(x, d).
\]
\end{lemma}

\begin{lemma} \label{lemma-four}
Assume Conjectures A and B. If
\[
\frac{\log x}{2 (\log\log x)^2}
 \leq p_{k}^{\sharp}
 \leq \frac{2 \log x}{\log\log x},
\]
then for $x$ sufficiently large and $p_{k}^{\sharp} < d < \min(p_{k + 1}^{\sharp}, o(\log x))$, we have
\[
N(x, p_{k}^{\sharp})
 > N(x, d).
\]
\end{lemma}

\begin{lemma} \label{lemma-five}
Assume Conjectures A and B. For any $0 < \delta' \leq 1 / 4$, let
\[
p_{k}^{\sharp}
 = \left\lceil \frac{(1+ \delta') \log x}{(\log\log x)^2} \right\rceil_{p_{j}^{\sharp}}.
\]
Then for $x$ sufficiently large and $d > p_{k}^{\sharp}$, we have
\[
N(x, p_{k}^{\sharp})
 > N(x, d).
\]
\end{lemma}

In what follows, it will be useful to introduce, in view of the representation in part (ii) of Lemma \ref{lemma-one}, the function
\[
M(x, d)
 = \mathfrak{S}(d) \left(1 - \frac{d}{\log x}\right),
\]
which increases in $d = p_{k}^{\sharp}$ until the factor $(1 - d / \log x)$ starts to decay rapidly enough. Let us note that by \eqref{equation-eight-pkloglogx}, where $(\log x)^{1 - o(1)} \leq p_{k}^{\sharp} \leq o(\log x)$,
\begin{align*}
\frac{M(x, p_{k}^{\sharp})}{M(x, p_{k - 1}^{\sharp})}
 &= \left(1 + \frac{1}{p_{k} - 2}\right) \left(1 - \frac{p_{k}^{\sharp}}{\log x}\right) \left(1 - \frac{p_{k - 1}^{\sharp}}{\log x}\right)^{-1} \\
 &\sim \left(1 + \frac{1}{\log\log x}\right) \left(1 - \frac{p_{k}^{\sharp}}{\log x}\right) \left(1 + \frac{p_{k - 1}^{\sharp}}{\log x}\right) \\
 &\sim 1 + \frac{1}{\log\log x} - \frac{p_{k}^{\sharp}}{\log x}, \quad \mbox{as $x \to \infty$}.
\end{align*}
If $p_{k}^{\sharp} = c \log x / \log\log x$, then this expression is greater than 1 when $c < 1$. It is less than 1 when $c > 1$. Thus, the transition from $M(x, d)$ increasing to decreasing occurs here.

In the proofs of Lemmas \ref{lemma-three}, \ref{lemma-four} and \ref{lemma-five}, we find it convenient to express Lemma \ref{lemma-one} in terms of the function $M(x, d)$. From \eqref{equation-two} and \eqref{equation-three}, we see that
\begin{equation} \label{equation-nine-N(x, d)-order}
N(x, d)
 \asymp \mathfrak{S}(d) \frac{x}{(\log x)^{2}},
\end{equation}
if $2 \leq d \leq o(\log x)$. Thus, \eqref{equation-three} now becomes
\begin{equation} \label{equation-ten-M(x, d)}
M(x, d) \frac{x}{(\log x)^{2}}
 = N(x, d) \left[1 + o\left(\frac{d}{\log x}\right) + o\left(\frac{1}{(\log\log x)^2}\right)\right],
\end{equation}
if $2 \leq d \leq o(\log x)$. Hence, we have
\begin{equation}\label{equation-eleven-Mbound}
M(x, d) \frac{x}{(\log x)^{2}}
 = N(x, d) \left[1 + o\left(\frac{1}{\log\log x}\right)\right],
\end{equation}
if $d = O(\log x / \log\log x)$.

In the proof of Lemma \ref{lemma-three}, we only require the error term $o(1 / \log\log x)$ in \eqref{equation-two} and \eqref{equation-three}. Hence, Lemma \ref{lemma-three} only requires the error $o(1 / \log\log x)$ in Conjecture A.

\begin{proof}[Proof of Lemma \ref{lemma-three}]
By \eqref{equation-eight-pkloglogx}, we have
\begin{align*}
\mathfrak{S}(p_{k - 1}^{\sharp})
 &= M(x, p_{k}^{\sharp}) \left(1 + \frac{1}{p_{k} - 2}\right)^{-1} \left(1 - \frac{p_{k}^{\sharp}}{\log x} \right)^{-1} \\
 &= M(x, p_{k}^{\sharp}) \left[1 + \left(\frac{p_{k}^{\sharp}}{\log x} - \frac{1}{\log\log x}\right) (1 + o(1)) \right],
\end{align*}
if $(\log x)^{1-o(1)}\le p_{k}^{\sharp} \leq o(\log x)$. By the choice of $p_{k}^{\sharp}$, we have $p_{k}^{\sharp} \leq (1 - \delta) \log x / \log \log x$, so that
\[
\mathfrak{S}(p_{k - 1}^{\sharp})
 \leq M(x, p_{k}^{\sharp}) \left(1 -\frac{\delta}{\log\log x} (1 + o(1))\right).
\]
Since $\mathfrak{S}(d) \leq \mathfrak{S}(p_{k - 1}^{\sharp})$ when $d < p_{k}^{\sharp}$, we obtain from \eqref{equation-two}, \eqref{equation-eleven-Mbound} and the last inequality that
\begin{align*}
N(x, d)
 &\leq \mathfrak{S}(d) \frac{x}{(\log x)^{2}} \left[1 + o\left(\frac{1}{(\log\log x)^{2}}\right)\right] \\
 & \leq \mathfrak{S}(p_{k - 1}^{\sharp}) \frac{x}{(\log x)^{2}} \left[1 + o\left(\frac{1}{(\log\log x)^{2}}\right)\right] \\
 & \leq M(x, p_{k}^{\sharp})\frac{x}{(\log x)^{2}} \left[1 - \frac{\delta}{\log\log x} + o\left(\frac{1}{\log\log x} \right)\right] \\
 &< N(x, p_{k}^{\sharp}) \left(1 -\frac{\delta}{2\log\log x} \right) \\
 &< N(x, p_{k}^{\sharp}),
\end{align*}
if $x$ is sufficiently large. Thus, we have the lemma.
\end{proof}

\begin{proof}[Proof of Lemma \ref{lemma-four}]
Suppose first that $d$ is restricted to the range $p_{k}^{\sharp} < d < 2 p_{k}^{\sharp}$. It follows that $\mathfrak{S}(d) \leq \mathfrak{S}(p_{k - 1}^{\sharp})$, and so we have
\begin{align*}
M(x, d)
 &\leq \mathfrak{S}(p_{k - 1}^{\sharp}) \left(1 - \frac{d}{\log x}\right) \\
 &\leq \mathfrak{S}(p_{k - 1}^{\sharp}) \left(1 - \frac{p_{k}^{\sharp}}{\log x}\right) \\
 &= M(x,p_{k}^{\sharp})\left(1 + \frac{1}{p_{k} - 2}\right)^{-1} \\
 &= M(x,p_{k}^{\sharp})\left(1 - \frac{1}{\log\log x}(1 + o(1))\right),
\end{align*}
using \eqref{equation-eight-pkloglogx} in the last step. Applying \eqref{equation-eleven-Mbound} twice for $x$ sufficiently large and \eqref{equation-nine-N(x, d)-order} for $\log x / (\log\log x)^2 \ll d \ll \log x / \log\log x$, we obtain from the above equation
\begin{align*}
N(x, d) \left[1 + o\left(\frac{1}{\log\log x}\right)\right]
 &= M(x, d)\frac{x}{(\log x)^{2}} \\
 &\leq N(x, p_{k}^{\sharp}) \left(1 - \frac{1}{\log\log x} (1+o(1))\right).
\end{align*}
Then applying \eqref{equation-nine-N(x, d)-order} once again, we find that
\[
N(x, d)
 \leq N(x, p_{k}^{\sharp}) \left(1 - \frac{1}{2 \log\log x}\right)
 < N(x, p_{k}^{\sharp}),
\]
which proves the lemma in the range $p_{k}^{\sharp} < d < 2 p_{k}^{\sharp}$. 

Next, we consider the range $2 p_{k}^{\sharp} \leq d < \min(p_{k + 1}^{\sharp}, o(\log x))$. Here, we need the full strength of the error term in Conjecture A. Furthermore, we can only make use of the inequality $\mathfrak{S}(d) \le \mathfrak{S}(p_{k}^{\sharp})$, since this becomes an equality when $d = m p_{k}^{\sharp}$ and $2 \leq m < p_{k + 1}$. In this range, we see that $p_{k}^{\sharp} \leq d / 2$. Hence, we have
\begin{align*}
M(x, d)
 &\leq \mathfrak{S}(p_{k}^{\sharp}) \left(1-\frac{d}{\log x}\right) \\
 &\leq \mathfrak{S}(p_{k}^{\sharp}) \left(1 - \frac{p_{k}^{\sharp}}{\log x} - \frac{d}{2 \log x} \right) \\
 &= M(x, p_{k}^{\sharp}) - \frac{d}{2 \log x} \mathfrak{S}(p_{k}^{\sharp}).
\end{align*}

Applying \eqref{equation-ten-M(x, d)} twice, we obtain
\begin{align*}
&N(x, d) \left[1 + o\left(\frac{d}{\log x}\right) + o\left(\frac{1}{(\log\log x)^2}\right)\right] \\
 &\hspace{1cm} \leq \left(M(x, p_{k}^{\sharp}) - \frac{d}{2 \log x} \mathfrak{S}(p_{k}^{\sharp})\right) \frac{x}{(\log x)^2} \\
&\hspace{1cm} = N(x, p_{k}^{\sharp}) \left[1 + o\left(\frac{p_{k}^{\sharp}}{\log x}\right) + o\left(\frac{1}{(\log\log x)^{2}}\right)\right] - \frac{d}{2\log x} \mathfrak{S}(p_{k}^{\sharp}) \frac{x}{(\log x)^2}.
\end{align*}
From \eqref{equation-nine-N(x, d)-order}, we see that
\begin{align*}
\frac{d}{\log x} N(x, d)
 &\asymp \frac{d}{\log x} \mathfrak{S}(d) \frac{x}{(\log x)^{2}} \\
 &\leq \frac{d}{\log x} \mathfrak{S}(p_{k}^{\sharp}) \frac{x}{(\log x)^{2}},
\end{align*}
if $p_{k}^{\sharp} \leq d < \min(p_{k + 1}^{\sharp}, o(\log x))$. Now, we also observe that $d / \log x \geq p_{k}^{\sharp} / \log x \gg 1 / (\log\log x)^{2}$. Therefore, for $x$ sufficiently large, we have
\begin{align*}
N(x, d)
 &\leq N(x, p_{k}^{\sharp}) - \frac{d}{2 \log x} \mathfrak{S}(p_{k}^{\sharp}) \frac{x}{(\log x)^{2}} (1 + o(1)) \\
 &\leq N(x, p_{k}^{\sharp}) - \frac{d}{4 \log x} \mathfrak{S}(p_{k}^{\sharp}) \frac{x}{(\log x)^{2}} \\
 &< N(x, p_{k}^{\sharp}),
\end{align*}
and the proof of the lemma is completed.
\end{proof}

\begin{proof}[Proof of Lemma \ref{lemma-five}]
By the choice of $p_{k}^{\sharp}$, we have
\[
\frac{(1 + \delta') \log x}{(\log\log x)^{2}}
 \leq p_{k}^{\sharp}
 \leq \frac{(1 + \delta') \log x}{\log\log x} (1 + o(1)),
\]
and therefore Lemma \ref{lemma-four} applies.

We proceed in the following way. We confine $d$ first to the range $ p_{k}^{\sharp} < d \leq \log x / (\log\log\log x)^{1 / 2}$. By means of Lemma \ref{lemma-four}, we have Lemma \ref{lemma-five} except for the range $p_{k + 1}^{\sharp} \leq d \leq \log x / (\log\log\log x)^{1 / 2}$, which could be empty. Here, we have $p_{k + 1}^{\sharp} > (1 + \delta') \log x / \log\log x$, so that $p_{k + 2}^{\sharp} > (1 + \delta') \log x (1 - o(1))$. However, this is outside the prescribed range of $d$. Hence, we have $p_{k + 1}^{\sharp} \leq d < p_{k + 2}^{\sharp}$, and so $\mathfrak{S}(d) \leq \mathfrak{S} (p_{k + 1}^{\sharp})$ in this range. From \eqref{equation-three} and \eqref{equation-eight-pkloglogx}, we thus have
\begin{align*}
N(x, d)
 &= \mathfrak{S}(d) \left(1 - \frac{d}{\log x} (1 + o(1))\right) \frac{x}{(\log x)^{2}} \\
 &\leq \mathfrak{S}(p_{k + 1}^{\sharp}) \left(1 - \frac{p_{k + 1}^{\sharp}}{\log x} (1 + o(1))\right) \frac{x}{(\log x)^{2}} \\
 &= M(x, p_{k}^{\sharp}) \left(1 + \frac{1}{p_{k + 1} - 2}\right) \left(1 - \frac{p_{k}^{\sharp}}{\log x}\right)^{-1} \\ &\hspace{5cm} \times \left(1 - \frac{p_{k + 1}^{\sharp}}{\log x} (1 + o(1))\right) \frac{x}{(\log x)^{2}} \\ &= M(x, p_{k}^{\sharp}) \left[1 - \left(\frac{p_{k + 1}^{\sharp}}{\log x} - \frac{1}{\log\log x}\right) (1 + o(1))\right] \frac{x}{(\log x)^{2}} \\
 &< M(x, p_{k}^{\sharp}) \left(1 - \frac{\delta'}{2 \log\log x}\right) \frac{x}{(\log x)^{2}} \\
 &= N(x, p_{k}^{\sharp})\left[1 + o\left(\frac{1}{\log\log x}\right)\right] \left(1 - \frac{\delta'}{2 \log\log x}\right) \\
 &< N(x, p_{k}^{\sharp}),
\end{align*}
where we used \eqref{equation-eleven-Mbound} in the penultimate step. Hence, we have the lemma in the first range for $d$.

Our next step is to restrict $d$ to the range $\log x / (\log\log\log x)^{1 / 2} \leq d \leq (\log x)^{2}$. If we take $D = \log x / (\log\log\log x)^{1 / 2}$ in \eqref{equation-four}, we obtain
\[
N(x, d)
 \leq \mathfrak{S}(d) \frac{x}{(\log x)^{2}} \left[1 - \frac{1}{(\log\log\log x)^{1 / 2}} (1 + o(1)) + O\left(\frac{1}{\log\log\log x}\right)\right].
\]
Now, a little thought discloses
\begin{align*}
\mathfrak{S}(d)
 &\leq \mathfrak{S}(p_{k}^{\sharp}) \frac{\mathfrak{S}(\lfloor (\log x)^{2} \rfloor_{p_{j}^{\sharp}})}{\mathfrak{S}(p_{k}^{\sharp})} \\
 &\leq \mathfrak{S}(p_{k}^{\sharp}) \left[1 + O\left(\frac{1}{\log\log\log x}\right)\right],
\end{align*}
the upper bound being derived from Mertens's formula in much the same manner that we established expression (3.9) in Goldston and Ledoan \cite{GoldstonLedoan2011}. Hence, we have
\[
N(x, d)
 \leq \mathfrak{S}(p_{k}^{\sharp}) \frac{x}{(\log x)^{2}} \left(1 - \frac{1}{(\log\log\log x)^{1 / 2}} (1 + o(1))\right).
\]

However, from \eqref{equation-three}, we have
\begin{align*}
N(x, p_{k}^{\sharp})
 &= \mathfrak{S}(p_{k}^{\sharp}) \frac{x}{(\log x)^{2}} \left[1 - \frac{p_{k}^{\sharp}}{\log x} (1 + o(1)) + o\left(\frac{1}{(\log\log x)^{2}}\right)\right] \\
 &\geq \mathfrak{S}(p_{k}^{\sharp}) \frac{x}{(\log x)^{2}} \left(1 - \frac{(1 + \delta')}{\log\log x} (1 + o(1))\right) \\
 &\geq \mathfrak{S}(p_{k}^{\sharp}) \frac{x}{(\log x)^{2}} \left(1 - \frac{2}{\log\log x}\right).
\end{align*}
Hence, we obtain the lemma in the second range for $d$ from these last two inequalities. 

Finally, we turn our attention to the range $d\geq (\log x)^{2}$. Using the estimate
\begin{align*}
N(x, H)
 &\leq \sum_{\substack{p_{n} \leq x \\ p_{n} - p_{n - 1} \geq H}} 1 \\
 &\leq \sum_{\substack{p_{n} \leq x \\ p_{n} - p_{n - 1} \geq H}} \frac{(p_{n} - p_{n - 1})}{H} \\
 &\leq \frac{1}{H}\sum_{p_{n} \leq x} (p_{n} - p_{n - 1}) \\
 &\leq \frac{x}{H},
\end{align*}
we find that
\[
N(x, d)
 \leq \frac{x}{(\log x)^{2}},
\]
if $d\geq (\log x)^{2}$. The lemma now follows from the lower bound
\[
N(x, p_{k}^{\sharp})
 \gg \frac{x}{(\log x)^{2}} \log\log\log x,
\]
if $(\log x)^{1 - o(1)} \leq p_{k}^{\sharp} \leq o(\log x)$.

It is straightforward to establish the last estimate. We apply \eqref{equation-nine-N(x, d)-order} to obtain
\[
N(x, p_{k}^{\sharp})
 \gg \mathfrak{S}( p_{k}^{\sharp}) \frac{x}{(\log x)^{2}},
\]
and deduce from Mertens's formula and \eqref{equation-eight-pkloglogx} that
\begin{align*}
\mathfrak{S}(p_{k}^{\sharp})
 &\gg \exp\left(\sum_{p \leq p_{k}} \frac{1}{p} + O(1)\right) \\
 &\gg \exp(\log\log p_{k} + O(1)) \\
 &\gg \log p_{k} \\
 &\sim \log\log\log x, \quad \mbox{as $x \to \infty$}.
\end{align*}
Assembling these last two results, we obtain the desired lower bound.
\end{proof}

We take up now the proof of the theorem.

\begin{proof}[Proof of the Theorem]
Let $a$, $a'$, $b$ and $b'$ be any fixed numbers satisfying $3 / 4 \leq a' < b < 1 < a < b' \leq 5 / 4$ and let $x$ be sufficiently large. Then from \eqref{equation-eight-pkloglogx}, we see that the interval
\[
\left[\frac{a' \log x}{(\log\log x)^{2}}, \frac{b' \log x}{\log\log x}\right]
\]
contains one to two primorials.

On the one hand, if the subinterval
\[
\left[\frac{a \log x}{(\log\log x)^{2}}, \frac{b \log x}{\log\log x}\right]
\]
contains a primorial, say $p_{k}^{\sharp}$, then we have
\[
p_{k}^{\sharp}
 = \left\lfloor \frac{b \log x}{\log\log x} \right\rfloor_{p_{j}^{\sharp}}
 = \left\lceil \frac{a \log x}{\log\log x)^{2}} \right\rceil_{p_{j}^{\sharp}}.
\]
By means of Lemma \ref{lemma-three}, we see that $N(x, p_{k}^{\sharp}) > N(x, d)$ when $d < p_{k}^{\sharp}$. According to Lemma \ref{lemma-five}, we have $N(x, p_{k}^{\sharp}) > N(x, d)$ when $d > p_{k}^{\sharp}$. Therefore, we conclude that $p_{k}^{\sharp}$ is the jumping champion for $x$.

On the other hand, if the subinterval
\[
\left[\frac{a \log x}{(\log\log x)^{2}}, \frac{b \log x}{\log\log x}\right]
\]
does not contain a primorial, then the two subintervals
\[
\left[\frac{a' \log x}{(\log\log x)^{2}}, \frac{a \log x}{(\log\log x)^{2}}\right)
\quad \mbox{and} \quad
\left(\frac{b \log x}{\log\log x}, \frac{b' \log x}{\log\log x}\right]
\]
must contain primorials. This is because each of the intervals
\[
\left[\frac{a' \log x}{(\log\log x)^{2}}, \frac{b \log x}{\log\log x}\right]
\quad \mbox{and} \quad
\left[\frac{a \log x}{(\log\log x)^{2}}, \frac{b' \log x}{\log\log x}\right]
\]
must contain a primorial.

Calling these consecutive primorials $p_{k}^{\sharp}$ and $p_{k + 1}^{\sharp}$, respectively, we write
\[
p_{k}^{\sharp}
 = \left\lfloor \frac{b \log x}{\log\log x} \right\rfloor_{p_{j}^{\sharp}}
\quad \mbox{and} \quad
p_{k + 1}^{\sharp}
 = \left\lceil \frac{a \log x}{(\log\log x)^{2}} \right\rceil_{p_{j}^{\sharp}}.
\]
By means of Lemma \ref{lemma-three}, we have again $N(x, p_{k}^{\sharp}) > N(x, d)$ when $d < p_{k}^{\sharp}$. From Lemma \ref{lemma-four}, we see that $N(x, p_{k}^{\sharp}) > N(x, d)$ when $p_{k}^{\sharp} < d <p_{k + 1}^{\sharp}$. Then using Lemma \ref{lemma-five}, we find that $N(x, p_{k + 1}^{\sharp}) > N(x, d)$ when $d > p_{k + 1}^{\sharp}$. Hence, the jumping champion for $x$ must be one of the two primorials $p_{k}^{\sharp}$ and $p_{k + 1}^{\sharp}$.

Furthermore, at the point of transition from one jumping champion to the next, there must be at least one value of $x$ for which there is a tie. Hence, there are two jumping champions for the same $x$. This completes the proof of the theorem.
\end{proof}

\section{Proof of Lemma \ref{lemma-two}}

The analysis performed by Odlyzko, Rubinstein and Wolf \cite{OdlyzkoRubinsteinWolf1999} demonstrates that in order to determine the precise point of transition between jumping champions one needs, among other things, accurate asymptotic formulas for the special type of singular series average
\[
A_{k}(d)
 = \sum_{1 \leq d_{1} < \ldots < d_{k - 2} < d} \mathfrak{S}(\{0, d_{1}, \ldots, d_{k - 2}, d\}).
\]

Letting
\[
A_{k}(d)
 = \mathfrak{S}(d) \frac{d^{k - 2}}{(k - 2)!} + R_{k}(d),
\]
Odlyzko, Rubinstein and Wolf proved that
\[
R_{k}(d)
 = O_{k}\left(\frac{d^{k - 2}}{\log\log d}\right).
\]
By a more elaborate method, we have improved this to
\[
R_{k}(d)
 = O_{k}(d^{k - 3 + \varepsilon}), \quad \mbox{for any $\varepsilon > 0$}.
\]
However, stronger results are probably true.

Odlyzko, Rubinstein and Wolf also presented numerical evidence that suggests that
\[
R_{k}(d)
 = O_{k}(\mathfrak{S}(d) d^{k - 3} \log d).
\]
Perhaps this can be replaced by an asymptotic formula.

By comparison, earlier work has been concerned with the average of the singular series over all the components of $\mathcal{D}$. For example, in 2004, Montgomery and Soundararajan \cite{MontgomerySoundararajan2004} proved that
\begin{align*}
\sum_{\substack{1 \leq d_{1}, \ldots, d_{k} \leq d \\ \text{distinct}}} \mathfrak{S}(\mathcal{D})
 &= d^{k} - \binom{k}{2} d^{k - 1} \log d + \binom{k}{2} (1 - \gamma - \log 2 \pi) d^{k - 1} \\ &\hspace{6cm} + O(d^{k - 3 / 2 + \varepsilon}),
\end{align*}
for a fixed $k \geq 2$, where $\gamma$ is Euler's constant. Their result strengthens the asymptotic formula originally proved by Gallagher \cite{Gallagher1976}, \cite{Gallagher1981} in 1976. It would be interesting to see if the method of Montgomery and Soundararajan can be made to apply to $A_{k}(d)$.

For our present purpose, we only require the simplest asymptotic result for the singular series averages $A_{k}(d)$ and $A_{k}(d, D)$ in Lemma \ref{lemma-two}. For this, we use an unpublished method of Ford \cite{Ford2007}, who gave a simple proof of the asymptotic formula above, but with the weaker error term $O(d^{k} / \log\log d)$. Ford's proof is based on the formula
\begin{equation} \label{equation-twelve-phiD}
\begin{split}
\phi(n, \mathcal{D})
 &= \sum_{\substack{1 \leq j \leq n \\ (Q_{\mathcal{D}}(j), n) = 1}} 1 \\
 &= n \prod_{p \mid n} \left(1 - \frac{\nu_{\mathcal{D}}(p)}{p}\right),
\end{split}
\end{equation}
where
\[
Q_{\mathcal{D}}(n)
 = \prod_{j = 1}^{k} (n + d_{j}).
\]
Here, we note that if $\abs{\mathcal{D}} = 1$ (so that $k = 1$), then \eqref{equation-twelve-phiD} reduces to the usual formula for the Euler totient function
\[
\phi(n)
 = n\prod_{p \mid n} \left(1 - \frac{1}{p}\right).
\]

To verify \eqref{equation-twelve-phiD}, we first detect the relatively prime condition with the M{\"o}bius function. We write
\begin{align*}
\phi(n, \mathcal{D})
 &= \sum_{1 \leq j \leq n} \sum_{\substack{d \mid n \\ d \mid Q_{\mathcal{D}}(j)}} \mu(d) \\
 &= \sum_{d \mid n} \mu(d) \sum_{\substack{1 \leq j \leq n \\ d \mid Q_{\mathcal{D}}(j)}} 1.
\end{align*}
The inner sum on $j$ on the far right-hand side is to be evaluated when $d$ is squarefree. By the Chinese remainder theorem, this sum is a multiplicative function in $d$ for squarefree $d$. For the case of $d = p$, we solve $p \mid Q_{\mathcal{D}}(j)$ by taking $j \equiv -d_{i} \pmod{p}$, where $1 \leq i \leq k$, among which we get $\nu_{\mathcal{D}}(p)$ distinct solutions for $j \pmod{p}$. For each of these distinct solutions, $j$ will run over a progression modulo $p$ with exactly $n / p$ terms in the range $1 \leq j \leq n$. Hence, for squarefree $d$, we have
\[
\sum_{\substack{1 \leq j \leq n \\ d \mid Q_{\mathcal{D}}(j)}}1
 = n \frac{\nu_{\mathcal{D}}(d)}{d},
\]
and on substituting we obtain the required result.

We now use \eqref{equation-twelve-phiD} to derive a useful expression for a truncated product of $\mathfrak{S}(\mathcal{D})$. Let 
\[
\mathfrak{S}_y(\mathcal{D})
 = \prod_{p \leq y} \left(1 - \frac{1}{p}\right)^{-k} \left(1 - \frac{\nu_{\mathcal{D}}(p)}{p}\right).
\]
Putting $p^{\sharp}(y) = \lfloor y\rfloor_{p_{j}^{\sharp}}$, from \eqref{equation-twelve-phiD} we deduce that
\begin{equation} \label{equation-thirteen-FordFormula}
\mathfrak{S}_{y} (\mathcal{D})
 = \frac{\phi(p^{\sharp}(y), \mathcal{D})}{p^{\sharp}(y)} \prod_{p \leq y} \left(1 - \frac{1}{p}\right)^{-k}.
\end{equation}

Next, we shall need the following elementary estimates. Suppose $1 \leq d_{1}, \ldots, d_{k} \leq h$. Then we have
\begin{equation} \label{equation-fourteen-Singular1}
\mathfrak{S}(\mathcal{D})
 = \left[1 + O_{k}\left(\frac{1}{\log\log h}\right)\right] \mathfrak{S}_{y}(\mathcal{D}),
\end{equation}
if $y \gg \log h$, and we have
\begin{equation} \label{equation-fifteen-Singular2}
\mathfrak{S}_{y}(\mathcal{D})
 = O_{k}((\log\log h)^{k - 1}),
\end{equation}
if $y = O(\log h)$.

To establish \eqref{equation-fourteen-Singular1}, we first note that the result is trivial if both singular series $\mathfrak{S}(\mathcal{D})$ and $\mathfrak{S}_{y}(\mathcal{D})$ are simultaneously equal to zero. In fact, these singular series will both be equal to zero or both not equal to zero when $y > k$. This is because, in order for either singular series to be zero, there must exist a prime $p$ for which $\nu_{\mathcal{D}}(p) = p$. Furthermore, since $\nu_{\mathcal{D}}(p) \leq k$, it is necessary that $p \leq k$. Therefore, we may assume that both singular series are not zero for the remainder of the proof.

We now let
\[
\Delta
 = \prod_{1 \leq i < j \leq k} \abs{d_{j} - d_{i}},
\]
and observe that $\nu_{\mathcal{D}}(p) = k$ if and only if $p \nmid \Delta$. Hence, we have
\[
\left(1 - \frac{1}{p}\right)^{-k} \left(1 - \frac{\nu_{\mathcal{D}}(p)}{p}\right)
 \ll \left\{ \begin{array}{ll}
   1 + O_{k}\left(\dfrac{1}{p}\right), & \mbox{if $p \mid \Delta$;} \\ [3ex]
   1 + O_{k}\left(\dfrac{1}{p^{2}}\right), & \mbox{if $p \nmid \Delta$.}
\end{array}
\right.
\]

On noting that $\Delta = O(h^{k^{2}})$ and $\omega(n) = O(\log n / \log\log n)$, where $\omega(n)$ denotes, as usual, the number of distinct prime factors of an integer $n$, we have $\omega(\Delta) = O_{k}(\log h / \log\log h)$. Therefore, we have
\begin{align*}
\frac{\mathfrak{S}(\mathcal{D})}{\mathfrak{S}_{y}(\mathcal{D})}
 &= \prod_{\substack{p \mid \Delta \\ p > y}} \left[1+ O_{k} \left(\frac{1}{p}\right)\right] \prod_{\substack{p \nmid \Delta \\ p > y}} \left[1 + O_{k} \left(\frac{1}{p^{2}}\right)\right] \\
 &= \exp\left(S_{I} + S_{II}\right),
\end{align*}
where 
\[
S_{I}
 = \sum_{\substack{p \mid \Delta \\ p > y}} \log \left[1 + O_{k} \left(\frac{1}{p}\right)\right]
 = O_{k}\left(\sum_{\substack{p \mid \Delta \\ p > y}}\frac{1}{p}\right)
 = O_{k}\left(\frac{\omega(\Delta)}{y}\right)
 = O_{k}\left(\frac{\log h}{y \log\log h}\right)
\]
and
\[
S_{II}
 = \sum_{\substack{p \nmid \Delta \\ p > y}} \log \left[1 + O_{k} \left(\frac{1}{p^{2}}\right)\right]
 = O_{k}\left( \sum_{p > y}\frac{1}{p^{2}}\right)
 = O_{k}\left( \frac{1}{y}\right).
\]
Hence, if $y \gg \log h$, we have
\begin{align*}
\frac{\mathfrak{S}(\mathcal{D})}{\mathfrak{S}_{y}(\mathcal{D})}
 &= \exp\left(O_{k}\left(\frac{1}{\log\log h}\right)\right) \\
 &= 1 + O_{k}\left(\frac{1}{\log\log h}\right),
\end{align*}
which proves \eqref{equation-fourteen-Singular1}.

To obtain \eqref{equation-fifteen-Singular2}, we simply note that $\nu_{\mathcal{D}}(p) \geq 1$ and deduce from Mertens's formula that
\begin{align*}
\mathfrak{S}_{y} (\mathcal{D})
 &\leq \prod_{p \leq y} \left(1 - \frac{1}{p}\right)^{1 - k} \\
 &= O_{k}\left(\exp\left((k - 1) \sum_{p \leq y}\frac{1}{p}\right)\right) \\
 &= O_{k}(\exp\left((k - 1) \log\log y\right)) \\
 &= O_{k}((\log\log h)^{k - 1}),
\end{align*}
if $y = O(\log h)$. Note that if we take $y=\log h$ in \eqref{equation-fourteen-Singular1} and \eqref{equation-fifteen-Singular2}, we obtain the sharp estimate
\begin{equation} \label{equation-sixteen-Singulargoodbound}
\mathfrak{S}(\mathcal{D})
 = O_{k}((\log\log h)^{k - 1}).
\end{equation}

For the proof of Lemma \ref{lemma-two}, we require a preliminary result.

\begin{lemma} \label{lemma-six}
For $\mathcal{D} \subset [0, h]$, $H \leq h$ and any $\varepsilon > 0$, we have
\[
\sum_{1 \leq d_{0} \leq H} \mathfrak{S}(\mathcal{D} \cup \{d_{0}\})
 = \mathfrak{S}(\mathcal{D}) H\left[1 + O_{k}\left(\frac{1}{\log\log H}\right)\right] + O(h^{\varepsilon}).
\]
\end{lemma}

\begin{proof}
We set $y = \log h$ and apply \eqref{equation-fourteen-Singular1} and \eqref{equation-thirteen-FordFormula}, in this order, to obtain
\begin{align*}
\sum_{1 \leq d_{0} \leq H} \mathfrak{S}(\mathcal{D} \cup \{d_{0}\})
 &= \left[1 + O_{k} \left(\frac{1}{\log\log H}\right)\right] \sum_{1 \leq d_{0} \leq H} \mathfrak{S}_{y} (\mathcal{D} \cup \{d_{0}\}) \\
 &= \left[1 + O_{k} \left(\frac{1}{\log\log H}\right)\right] \frac{1}{p^{\sharp}(y)} \prod_{p \leq y} \left(1-\frac{1}{p}\right)^{-k - 1} \\ &\hspace{4.5cm} \times \sum_{1 \leq d_{0} \leq H} \phi(p^{\sharp}(y), \mathcal{D} \cup \{d_{0}\}).
\end{align*}
It remains only to consider the sum on the far right-hand side.

We have
\begin{align*}
\sum_{1 \leq d_{0} \leq H} \phi(p^{\sharp}(y), \mathcal{D} \cup \{d_{0}\})
 &= \sum_{1 \leq d_{0} \leq H} \sum_{\substack{1 \leq j \leq p^{\sharp}(y) \\ (Q_{\mathcal{D} \cup \{d_{0}\}}(j), p^{\sharp}(y)) = 1}} 1 \\
 &= \sum_{\substack{1 \leq j \leq p^{\sharp}(y) \\ (Q_{\mathcal{D}}(j), p^{\sharp}(y)) = 1}} \sum_{\substack{1 \leq d_{0} \leq H \\ (j + d_{0}, p^{\sharp}(y)) = 1}} 1.
\end{align*}
To treat the inner sum on $d_{0}$ on the far right-hand side, we write
\begin{align*}
\sum_{\substack{1 \leq d_{0} \leq H \\ (j + d_{0}, p^{\sharp}(y)) = 1}} 1
 &= \sum_{d \mid p^{\sharp}(y)} \mu(d) \sum_{\substack{1 \leq d_{0} \leq H \\ d_{0} \equiv -j \pmod{d}}} 1 \\
 &= \sum_{d \mid p^{\sharp}(y)} \mu(d) \left(\frac{H}{d} + O(1)\right) \\
 &= H \prod_{p \leq y} \left(1 - \frac{1}{p}\right) + O(2^{\pi(y)}),
\end{align*}
where $\pi(y)$ gives the number of primes less than or equal to $y$.

Therefore, we have
\[
\sum_{1 \leq d_{0} \leq H} \phi(p^{\sharp}(y), \mathcal{D} \cup \{d_{0}\})
 = \phi(p^{\sharp}(y), \mathcal{D})\left[H \prod_{p \leq y} \left(1 - \frac{1}{p}\right) + O( 2^{\pi(y)})\right].
\]
Finally, from \eqref{equation-fourteen-Singular1} and \eqref{equation-fifteen-Singular2}, we see that
\begin{align*}
\sum_{1 \leq d_{0} \leq H} \mathfrak{S}(\mathcal{D} \cup \{d_{0}\})
 &= H \left[1 + O_{k}\left(\frac{1}{\log\log H}\right)\right] \frac{\phi(p^{\sharp}(y), \mathcal{D})}{p^{\sharp}(y)} \prod_{p \leq y} \left(1 - \frac{1}{p}\right)^{-k} \\
  &\hspace{3.5cm} + O\left(2^{\pi(y)} \prod_{p \leq y}\left(1 - \frac{1}{p}\right)^{-k - 1}\right) \\
 &= \mathfrak{S}_y(\mathcal{D}) H \left[1 + O_{k}\left(\frac{1}{\log\log H}\right)\right] + O(h^{\varepsilon}) \\
 &= \mathfrak{S}(\mathcal{D}) H \left[1 + O_{k}\left(\frac{1}{\log\log H}\right)\right] + O(h^{\varepsilon}).
 \end{align*}
This finishes the proof of the lemma.
\end{proof}

We take up now the postponed proof of Lemma \ref{lemma-two}.

\begin{proof}[Proof of Lemma \ref{lemma-two}]
We have
\begin{align*}
A_{k}(d, D)
 &= \frac{1}{(k-2)!} \sum_{1 \leq d_{1}, \ldots, d_{k - 2} < D} \mathfrak{S}(\{0, d_{1}, \ldots, d_{k - 2}, d\}) \\ &\hspace{5cm} + O_{k}(D^{k - 3}(\log\log h)^{k - 1}).
\end{align*}
This is because there are $(k - 2)!$ permutations of the summands when they are ordered by the inequalities in the definition of $A_{k}(d, D)$. The error term comes from including the terms when one or more of the summands are equal and applying \eqref{equation-sixteen-Singulargoodbound}. Then the required result follows by applying Lemma \ref{lemma-six} repeatedly to the main term above.
\end{proof}

\bibliographystyle{amsplain}

\end{document}